\documentclass[11pt]{amsart}
\usepackage{amsmath, amsthm, amssymb, amsfonts, verbatim}
\usepackage[pdftex]{graphicx}
\usepackage[bookmarks]{hyperref}
\usepackage{color}
\reversemarginpar

\newtheorem{thm}{Theorem}[section]
\newtheorem{cor}[thm]{Corollary}
\newtheorem{lem}[thm]{Lemma}

\theoremstyle{definition}

\theoremstyle{remark}
\newtheorem{remark}[thm]{Remark}

\numberwithin{equation}{section}
\numberwithin{thm}{section}
\newtheorem*{acknowledgment}{Acknowledgments} 

\newcommand{\ip}[1]{\langle{#1}\rangle}

\newcommand{\norm}[1]{\left\Vert#1\right\Vert}
\newcommand{\abs}[1]{\left\vert#1\right\vert}
\newcommand{\R}{{\mathbb{R}}}

\newcommand{\lap}{\Delta}
\newcommand{\pd}{\partial}

\newcommand{\curl}{\mbox{curl}}
\newcommand{\dv}{\mbox{div}}

\DeclareMathOperator*{\tr}{tr}

\def\qand{\quad \mbox{and}\quad}
\def\be{\begin{equation}}
\def\ee{\end{equation}}

\def\a{\alpha}
\def\b{\beta}

\begin{document}
\title[Interaction Morawetz for mNLS]{Interaction Morawetz estimate for the magnetic Schr\"odinger equation and applications}
\author[J. Colliander]{James Colliander}
%\thanks{J.C. was supported in part by NSERC grant RGP250233-07}
\address{Department of Mathematics, University of Toronto}
\email{colliand@math.toronto.edu}
\author[M. Czubak]{Magdalena Czubak}
\address{Department of Mathematical Sciences, Binghamton University (SUNY)}
\email{czubak@math.binghamton.edu}

\author[J. Lee]{Jeonghun Lee}
\address{School of Mathematics, University of Minnesota}
\email{leex2454@umn.edu}
\date{\today}
\maketitle
\begin{abstract}
We establish an interaction Morawetz estimate for the magnetic Schr\"odinger equation for $n\geq 3$ under certain smallness conditions on the gauge potentials, but with almost optimal decay.  As an application, we prove global wellposedness and scattering in $H^{1}$ for the cubic defocusing magnetic Schr\"odinger equation for $n=3$.
\end{abstract}
\section{Introduction}
The purpose of this article is to  study the interaction Morawetz estimates for the magnetic Schr\"odinger equation.  Morawetz type estimates have their origins in \cite{Morawetz} and \cite{LinStrauss}.   The first \emph{interaction} Morawetz estimate was established for the cubic defocusing NLS \cite{CKSTT04}, and it reads as follows
\be\label{IM1}
\int^{T}_{0}\int_{\R^{3}}\abs{u(t,x)}^{4}dx dt \lesssim \norm{u(0)}^{2}_{2}\sup_{[0,T]} \norm{u(t)}^{2}_{\dot H^{\frac 12}_{x}}.
\ee
In particular, it allowed for a simpler proof of scattering obtained previously in \cite{GinibreVelo85}.
The estimate was extended to $n\geq 4$ in \cite{RyckmanVisan, Visan07} giving
\be\label{IM2}
\norm{|\nabla|^{-\frac{n-3}{2}}(\abs{u(t,x)}^{2})}^{2}_{L^{2}([0,T]\times R^{n})}\lesssim  \norm{u(0)}^{2}_{2}\sup_{[0,T]} \norm{u(t)}^{2}_{\dot H^{\frac 12}_{x}}.
\ee
Then building on an idea of Hassell and other advances, a new proof was obtained in \cite{CGT08} that applies to all dimensions $n\geq 1$.  An independent proof was also achieved in \cite{PlanchonVega}.  For a more detailed background on Morawetz type estimates we refer to \cite{CGT08, GinibreVelo10}.
\newline\indent
Now let $n\geq 3$ and consider the \emph{magnetic} nonlinear Schr\"odinger equation 
\be\tag{mNLS}
\begin{split}
iD_{t}u+ \lap_{A} u&=\mu g (\abs{u}^{2})u,\\
u(0)&=u_{0},
\end{split}
\ee
where 
\begin{align*}
&u: \R^{n+1}\mapsto \mathbb C,\\
&A_{\alpha}: \R^{n}\mapsto \R,\quad \alpha=0,\cdots, n,\\
&D_\alpha =\partial_\alpha + iA_\alpha,\quad \alpha=0,\cdots, n, \quad D_{t}=D_{0},\\
&\lap_{A}=D^{2}=D_{j}D_{j}=(\partial_j + iA_j)(\partial_j + iA_j)=\lap+iA\cdot\nabla+i\nabla\cdot A-\abs{A}^{2},\\
&g (r) = r^p, \quad r \geq 0, p >0.
\end{align*}
We use the standard notation, that the greek indices range from $0$ to $n$, and Roman indices range from $1$ to $n$.  We also sum over repeated indices.
The case of $\mu=1$ is usually called defocusing and $\mu=-1$ is called focusing.  We suppose we are in the Coulomb gauge, $\nabla \cdot A=0$.   The main result is
\begin{thm}\label{thm1} Let $n\geq 3,$ and let $u$ solve the defocusing mNLS.  Suppose $(A_{0}, A)$ satisfy \eqref{FVc0}- \eqref{FVc2} and \eqref{latest}-\eqref{latestc4}.  Then the following estimate holds
\be\label{im}
\norm{|\nabla|^{-\frac{n-3}{2}}(\abs{u}^{2}) }^{2}_{L^{2}([0,T]\times \R^{n})}\lesssim \norm{u_{0}}^{2}_{L^{2}_{x}}\sup_{[0,T]}\norm{(-\lap_A)^{\frac 14} u}^{2}_{L^{2}_{x}} .
\ee
\end{thm}
\noindent The conditions on the gauge potentials $(A_0, A)$ will soon be discussed in more detail in Section \ref{A_condtions}.  As an application we show
\begin{thm}\label{thm3}  Let $(A_{0}, A)$ satisfy \eqref{FVc0}-\eqref{FVc2}, \eqref{latest}-\eqref{latestc4} and \eqref{FVc3}-\eqref{FVc6}. Then for given initial data in $H^{1}(\R^{3})$, mNLS with a defocusing cubic nonlinearity is globally wellposed and scatters to the linear magnetic Schr\"{o}dinger equation.
\end{thm}
While the theory of existence and uniqueness has been considered
before for the nonlinear mNLS (see \cite{CazenaveEsteban88,
DeBouard91, NakamuraShimomura05, Laurent08}) this is the first result (to our
knowledge) on scattering for the nonlinear equation.  Scattering for the one particle mNLS without the nonlinearity has been considered by many authors.  We refer the reader to \cite{LossThaller87, Robert92, AHS78, Laba93, Iwashita95} and references therein.
\begin{remark}

We also would like to note that in the proof of local well-posedness we do not reproduce the same contraction argument usually done for the cubic NLS (see for example \cite{Taobook}).  The reason for this is that even though we have Strichartz estimates for the mNLS, to fully benefit from them we would either need to extend estimate \eqref{comp4} below to other $L^p$ besides $L^2$ or establish a variant of a product rule in $L^p$.  Doing that does not seem easier than using alternate Strichartz exponents.  The current approach has a flavor of what is usually done for the critical equations, and hence it is more involved than if we were going to just use the standard subcritical methods.
\end{remark}
The proof of Theorem \ref{thm1} relies on the \emph{commutator vector operators} method developed in \cite{CGT08}, where it was used to obtain the interaction Morawetz estimate for the classical NLS.  We show in this article that the method is robust and can be extended to the magnetic case.  

The main ingredient comes from the local conservation laws. In the case of the classical NLS the momentum is conserved.  In the case of mNLS, we obtain only a balance law (see \eqref{p2} and \eqref{t00}-\eqref{tjk} for precise definitions)
\[
\partial_{t}T_{0j}+\partial_{k}T_{jk}=F_{\a j}T_{\a0},
\]
which eventually results in a need to control a term of the form $$B(t)=\int_{\R^{n}}\int_{\R^{n}}\frac {x_{j}-y_{j}}{\abs{x-y}}F_{\a j}(t,x)T_{\a0}(t,x)\abs{u(t,y)}^{2} dx dy.$$  If $B(t)$ were positive, we could just ignore it (see for example the proof of \eqref{virial3}).  However as shown in the appendix, this cannot be expected in general. Another way to handle this term follows the path used by Fanelli and Vega \cite{FanelliVega09} for the linear magnetic Schr\"odinger equation.  Moreover, applying the results of \cite{FanelliVega09}, D'Ancona, Fanelli, Vega and Visciglia established a family of Strichartz estimates \cite{JFA}.  Their work motivates us to assume similar conditions on the gauge potentials.  As a result, we can control the term $B(t)$, and hence obtain the interaction Morawetz estimates for mNLS.
\newline\indent To show Theorem \ref{thm3} we need an inhomogeneous Strichartz estimate.  This is a simple consequence of the Christ-Kiselev Lemma and Strichartz estimates from \cite{JFA} (stated in Theorem \ref{mss} and Theorem \ref{CK}), and we record it here only for completeness.
\begin{thm}\label{thm2}
Consider an inhomogeneous linear magnetic Schr\"{o}dinger equation with zero initial data on $\R^{1+n}, n\geq 3$. 
\be \label{inhomog_eq}
\begin{split}
i D_{t}u + \lap_{A}u = N, \\
u(0) = 0,
\end{split}
\ee
and suppose $u$ is a solution of (\ref{inhomog_eq}) and that $(A_{0}, A)$ satisfy \eqref{FVc0}-\eqref{FVc2}, \eqref{FVc3}-\eqref{FVc6}. Then 
\begin{align*}
\| u \|_{L_t^q L_x^r} \lesssim \| N \|_{L_t^{\tilde{q}'} L_x^{\tilde{r}'}} ,
\end{align*}
for Schr\"odinger admissible Strichartz pairs $(q, r), (\tilde{q}, \tilde{r})$ such that both admissible pairs satisfy
\begin{align*}
\frac 2q + \frac nr &= \frac n2, \quad  2 \leq q, \; q \not = 2 \text{ if } n=3 , \text{ and } \tilde q\not =2 \text { if } q=2 \text { for } n>3,
\end{align*}
and where $p'$ denotes the H\"older dual exponent of $p$. 
\end{thm}
The dispersive properties of the magnetic Schr\"odinger equations have been studied also by \cite{D'AnconaFanelli08, EGS08, EGS09, GST07, MMT08}.  We would like to investigate in the future if the interaction Morawetz estimates could be recaptured in the setting of these works.  Also see \cite{AFG11, Garcia11}.
\newline\indent   
The organization of the paper is as follows.  In Section \ref{prelim} we gather some identities and known estimates.  In Section \ref{cvi} we derive conservation laws and the generalized magnetic virial identity, which are then applied in Section \ref{interaction} to show Theorem \ref{thm1}.  In Section \ref{nh} we prove the inhomogeneous Strichartz estimate.  Section \ref{app} is devoted to the proof of Theorem \ref{thm3}.     

\begin{acknowledgment}
The authors would like to thank the referees for the helpful comments.
The first author was supported in part by NSERC grant RGP250233-07.  The second author was partially supported by a grant from the Simons Foundation \#246255.
\end{acknowledgment}

\section{Preliminaries}\label{prelim}
We start by stating the following identities, which are easily verified by a direct computation
\begin{align}
\partial_{\a}(u\bar v)&=(D_{\a}u)\bar v+u\overline{D_{\a}v}\label{p1},\\
D_{\a}D_{\b}&=iF_{\a\b}+D_{\b}D_{\a}, \ \mbox{where} \ F_{\a\b}=\partial_{\a}A_{\b}-\partial_{\b}A_{\a}\label{p2},\\
D_{\a}(uv)&=(D_{\a}u)v+u\partial_{\a}v\label{p3}.
\end{align}
We recall the standard Strichartz estimates \cite{GinibreVelo92, Yajima87, KeelTao98}.  If $(q, r)$ is Schr\"odinger admissible, i.e.,
\[
\frac {2}{q}+\frac nr=\frac n2,\ q\geq 2, \ \ q\neq 2 \ \mbox{if}\ n=2,
\]
then
\begin{align}
\norm{e^{it\Delta}\phi}_{L^{q}_{t}L^{r}_{x}}&\leq C\norm{\phi}_{L^{2}_{x}},\label{Str1}\\
\norm{\int_0^t e^{i(t-s)\Delta}N(s,\cdot)ds}_{L^{q}_{t}L^{r}_{x}}&\leq C\norm{N}_{L^{\tilde q'}_{t}L^{\tilde r'}_{x}}\label{Str2},
\end{align}
where $(\tilde q',\tilde r')$ are H\"older dual exponents of a Schr\"odinger admissible pair $(\tilde q, \tilde r)$.\\
\newline
We also use the following local smoothing estimate (see \cite{JFA} for historical remarks)
\begin{thm}\cite{RuizVega93}
If $(q,r)$ is Schr\"odinger admissible, then
\begin{align}\label{RVest}
\norm{|\nabla|^{\frac 12}\int^{t}_{0}e^{i(t-s)\Delta}N(s,\cdot)ds}_{L^{q}_{t}L^{r}_{x}}\lesssim \sum_{j\in \mathbb Z}2^{\frac j2}\norm{\chi_{j}N}_{L^{2}_{t}L^{2}_{x}},
\end{align}
where $\chi_{j}=\chi_{\{2^{j}\leq \abs{x}\leq 2^{j+1}\}}$.
\end{thm}
We now discuss the needed conditions on the gauge potentials.
\subsection{Conditions on the gauge potentials}\label{A_condtions}
The curvature, $F$, of the gauge potential $(A, A_{0})$ is a two-form given by
\be\label{Fdef}
F=\frac 12 F_{\alpha\beta} dx^{\alpha}\wedge dx^{\beta},
\ee
where $F_{\alpha\beta}$ is given in \eqref{p2}.  From $F$ we can extract the magnetic field $dA$ by only considering the spatial coordinates of $F$ in \eqref{Fdef}.  Similarly we can extract the electric field from the temporal-spatial components.
\newline\indent
In $3$ dimensions the magnetic field is often identified with a vector field $\curl A$.
It was observed in \cite{FanelliVega09} that the \emph{trapping component}, $B_{\tau}$, of the magnetic field given by
\[
B_{\tau}=\frac{x}{\abs{x}}\wedge \curl A,
\]
was an obstruction to the dispersion.  This can be thought of as the tangential component of the magnetic field with respect to the unit sphere.  In higher dimensions the trapping component can be rephrased as
\[
B^{T}_{\tau}=\frac{x^{T}}{\abs{x}} (F_{jk}),
\]
where  $(F_{jk})$ is a matrix with the $(j,k)$ entry given by $F_{jk}$.  Thus the $k$'th entry of the vector $B_{\tau}$ is $\frac{x_{j}}{\abs{x}} F_{jk}$.  
\newline\indent
Next, if we take the radial derivative of $A_{0}$ and decompose it into positive and negative parts
\[
\partial_{r}A_{0}=\left(\nabla A_{0}\cdot \frac{x}{\abs{x}}\right)_{+}- \left(\nabla A_{0}\cdot \frac{x}{\abs{x}}\right)_{-},
\]
then the positive part can also affect dispersion \cite{FanelliVega09}.  The conditions that were used in \cite{FanelliVega09} are
\begin{align}
&(A_{0},A)\in C^{1}_{loc}(\R^{n}\setminus\{0\}),\  \Delta_{A},\ H=-\Delta_{A}+A_{0} \ \mbox{are self adjoint and positive on} \ L^{2},\label{FVc0}\\
&\dv A=0,\label{FVc1}
\end{align}
\be\label{FVc2}
\begin{split}
&\mbox{if}\ n=3,\ 
\frac {(M+\frac 12)^{2}}{M}\norm{\abs{x}^{3/2}B_{\tau}}^{2}_{L^{2}_{r}L^{\infty}(S_{r})}+(2M+1)\norm{\abs{x}^{2}(\partial_{r}A_{0})_{+}}_{L^{1}_{r}L^{\infty}(S_{r})}<\frac 12,\\
&\mbox{if}\ n\geq 4, \ \norm{\abs{x}^{2}B_{\tau}}^{2}_{L^{\infty}_{x}}+2\norm{\abs{x}^{3}(\partial_{r}A_{0})_{+}}_{L^{\infty}_{x}}<\frac 23(n-1)(n-3),
\end{split}
\ee
for some $M>0$ (see \cite[Remark 1.3]{JFA}), and where \[\norm{f}^{p}_{L^{p}_{r}L^{\infty}(S_{r})}=\int^{\infty}_{0}\sup_{\abs{x}=r}\abs{f}^{p}dr.\]  
With those assumptions Fanelli and Vega were able to show some weak dispersion properties of the solutions of the linear mNLS \cite[Theorems 1.9 and 1.10]{FanelliVega09}.
The following is a part of Theorems 1.9 and 1.10 \cite{FanelliVega09}.  Note with $H=-\Delta_{A}+A_{0}$, linear mNLS, can be written as
\[
iu_t=Hu,
\]
so $e^{-itH}\phi$ below refers to the solution with initial data $u(0)=\phi$.

\begin{thm}\cite{FanelliVega09}\label{thmFV1.910} 
Let $\phi\in L^{2}, \lap_{A}\phi \in L^{2}$, $(A_{0}, A)$ satisfy \eqref{FVc0}-\eqref{FVc2}, and let $\nabla^{\tau}_{A}$ denote the projection of $\nabla_{A}$ on the tangent space to the unit sphere $\abs{x}=1$, $\nabla^{\tau}_{A}u=\nabla_{A}u-\frac{x}{\abs{x}}\left(\frac{x}{\abs{x}}\cdot \nabla_{A}u \right)$,
then
\begin{align}
\mbox{if}\ n=3,\quad &\int^{\infty}_{0}\int_{\R^{3}} \frac{\abs{\nabla^{\tau}_{A}e^{-itH}\phi}^{2}}{\abs{x}}dxdt+ \sup_{R>0}\frac {1}{R} \int^{\infty}_{0}\int_{\abs{x}\leq R}\abs{\nabla_{A}e^{-itH}\phi}^{2}dx dt\nonumber\\
&\qquad + \int_{0}^{\infty} \sup_{R>0}\frac {1}{R^{2}}\int_{\abs{x}=R} \abs{e^{-itH}\phi}^{2}d\sigma dt
\leq C \norm{(-\lap_A)^{\frac 14} \phi}^{2}_{L^{2}},\\
\mbox{if}\ n\geq 4,\quad &\int^{\infty}_{0}\int_{\R^{n}} \frac{\abs{\nabla^{\tau}_{A}e^{-itH}\phi}^{2}}{\abs{x}}dxdt+ \sup_{R>0}\frac 1R \int_{0}^{\infty}\int_{\abs{x}\leq R}\abs{\nabla_{A}e^{-itH}\phi}^{2}dx dt\nonumber\\
&\qquad +   \int_{0}^{\infty}\int_{\R^{n}} \frac{\abs{e^{-itH}\phi}^{2}}{\abs{x}^{3}}dx dt
\leq C \norm{(-\lap_A)^{\frac 14} \phi}^{2}_{L^{2}}.
\end{align}
 \end{thm}
Following the proof of Theorem \ref{thmFV1.910} one can establish analogs of these estimates for the \emph{nonlinear, defocusing} mNLS. 
\begin{cor}\label{nhcor}
With the same assumptions as in Theorem \ref{thmFV1.910} we have
\begin{align}
\mbox{if}\ n=3,\quad &\int^{T}_{0}\int_{R^{3}}\big( \frac{\abs{\nabla^{\tau}_{A}u}^{2}}{\abs{x}}+\frac{2MG(\abs{u}^{2})}{\abs{x}}\big)dxdt+  \sup_{R>0}\frac {1}{R}\int^{T}_{0}\int_{\abs{x}\leq R}\big(\abs{\nabla_{A}u}^{2}+G(\abs{u}^2)\big)dx dt\nonumber\\
&\qquad + \sup_{R>0}\frac {1}{R^{2}}\int_{0}^{T} \int_{\abs{x}=R} \abs{u}^{2}d\sigma dt
\leq C \sup_{t\in [0,T]}\norm{(-\lap_A)^{\frac 14} u(t)}^{2}_{L^{2}},\label{ey3}\\
\mbox{if}\ n\geq 4,\quad &\int^{T}_{0}\int_{R^{n}} \frac{\abs{\nabla^{\tau}_{A}u}^{2}}{\abs{x}}dxdt+\sup_{R>0}\frac 1R \int_{0}^{T} \int_{\abs{x}\leq R}\big(\abs{\nabla_{A}u}^{2}+\frac{n-1}{2}G(\abs{u}^2)\big)dx dt\nonumber\\
&\qquad +  \int_{0}^{T}\int_{\R^{n}}( \frac{\abs{u}^{2}}{\abs{x}^{3}}+(n-1)\frac{G(\abs{u}^2)}{\abs{x}})dx dt
\leq C\sup_{t\in[0,T]} \norm{(-\lap_A)^{\frac 14}u(t)}^{2}_{L^{2}},\label{ey4}
\end{align}
for any $T\in (0,\infty]$, 
where $u$ solves mNLS with a defocusing nonlinearity $g(\abs{u}^2)u$, $G\geq 0$ and satisfies $G'(x) = xg'(x),$ and $M$ is as in \eqref{FVc2}.
\end{cor}
This follows immediately from Theorems 1.9 and 1.10 in \cite{FanelliVega09} once we observe that the proofs of these theorems rely on the generalized virial identity.   The virial identity \cite[Theorem 1.2]{FanelliVega09} is for the homogeneous equation, but the addition of the defocusing nonlinearity leads to an addition of a term (see Lemma \ref{Virial} and Corollary \ref{vcor}) that is positive with $a$ as in \cite{FanelliVega09} and results in an identical proof as before.   
\subsubsection{Interaction Morawetz: curvature conditions}
In order to establish Theorem \ref{thm1} in addition to conditions \eqref{FVc0}-\eqref{FVc2} we impose the following (compare with \eqref{FVc2} and \eqref{FVc6} below).
Let
\be\label{Cj}
C_j=\{x: 2^{j}\leq \abs{x}\leq 2^{j+1}\}.
\ee
and we assume there is $0<b<1$ satisfying the following:
\begin{align}
\sum_{j \in \mathbb Z} 2^j\sup_{C_j}\abs{dA}^{2-2b}<\infty\label{latest}.
\end{align}
For $n=3$, 
\begin{align}
 & \| |dA|^b |x| \|_{L_r^2 L^\infty(S_r)} = \int^{\infty}_{0}\sup_{\abs{x}=r}\abs{x}^{2}\abs{dA}^{2b}dr<\infty,\label{latestc1}  \\
 &\norm{\abs{x}^{2}\nabla A_{0}}_{L^{1}_{r}L^{\infty}(S_{r})}=\int^{\infty}_{0}\sup_{\abs{x}=r}\abs{x}^{2}\abs{\nabla A_{0}}dr<\infty\label{latestc2},
\end{align}
and for $n\geq 4$
\begin{align}
  \norm{\abs{x}^{3} \abs{dA}^{2b}}_{L^{\infty}_{x}}&<\infty,\label{latestc3}  \\
 \norm{\abs{x}^{3} \nabla A_{0}}_{L^{\infty}_{x}}&<\infty\label{latestc4}.
\end{align}

\begin{remark}
Note that in comparison to \eqref{FVc2}, the assumptions are made on the whole curvature and not just the projected components.  On the other hand, we do not require the curvature to be \emph{small} in these norms as in  \eqref{FVc2}, but merely to be \emph{bounded}.  
In addition, the norms for the temporal component $F_{0j}=-\partial_{j}A_{0}$ are the same as \eqref{FVc2} whereas the magnetic field $dA$ is using now a slightly stronger norm.
\newline\indent
Finally, observe that the magnetic field $\abs{dA}\sim \frac{1}{\langle x\rangle^{2+\epsilon}}$ satisfies the conditions with $b=\frac 34$.  Such magnetic field  corresponds to $A$ decaying like $\frac{1}{\langle x\rangle^{1+\epsilon}}$. Similarly, $A_0\sim\frac{1}{\ip{x}^{2+\epsilon}}$ satisfies the needed conditions.  This type of decay for $(A_0, A)$ is almost optimal \cite{EGS09}. Hence Theorem \ref{thm1} implies interaction Morawetz estimates for potentials with almost optimal decay.
\end{remark}

\subsubsection{Inhomogeneous Strichartz estimate: gauge potential conditions}
Now, to establish the inhomogeneous Strichartz estimate, besides \eqref{FVc0}-\eqref{FVc2} we need additional conditions found in \cite{JFA}.  (We do not require here \eqref{latest}-\eqref{latestc4}.)   They are
\begin{align}
&\abs{A}^{2}-2i A \cdot \nabla + A_{0} \in L^{\frac n2,\infty}, \quad A \in L^{n,\infty},\label{FVc3}\\
&\norm{(A_{0})_{+}}_{K}<\infty \label{FVc4},\\
&\norm{( A_{0})_{-}}_{K}<\frac{\pi^{n/2}}{\Gamma(\frac n2-1)},\label{FVc5}\\
& \sum_{j\in \mathbb Z}2^{j} \sup_{x \in C_{j}}\abs{A}+\sum_{j\in \mathbb Z}2^{2j}\sum_{x\in C_{j}} \abs{A_{0}} <\infty,\label{FVc6}
\end{align}
where $\norm{\cdot}_{K}$ is the Kato norm defined by
\[
\norm{f}_{K}=\sup_{x\in \R^{n}}\int\frac{\abs{f(y)}}{\abs{x-y}^{n-2}}dy,
\]
and where $C_{j}$ is as in \eqref{Cj}.

\subsection{Magnetic Schr\"odinger Strichartz and other estimates used.}
\begin{thm}\label{compHD}\cite{JFA}
Let $n\geq 3$ and $H=-\Delta_{A}+A_{0}$.  Suppose $(A_{0}, A)$ satisfy \eqref{FVc0}-\eqref{FVc2} and \eqref{FVc3}-\eqref{FVc5},  then
\begin{align}
\norm{H^{\frac 14}\phi}_{L^{q}}&\leq C_{q}\norm{|\nabla|^{\frac 12}\phi}_{L^{q}}, \ 1<q<2n,\label{comp1}\\
\norm{H^{\frac 14}\phi}_{L^{q}}&\geq c_{q}\norm{|\nabla|^{\frac 12}\phi}_{L^{q}}, \ \frac 43<q<4\label{comp2}.
\end{align}
\end{thm}
As one consequence we have a boundedness of $H^{-\frac 14}(-\lap_{A})^{\frac 14}$ on $L^{2}_{x}$ as follows.  First apply \eqref{comp1} for an operator with $A_{0}=0$, and then \eqref{comp2} to get
\begin{align*}
&H^{-\frac 14}L^{2}_{x}\hookrightarrow (-\lap_{A})^{-\frac 14}L^{2}_{x},
\intertext{from which by duality we have,}
 &(-\lap_{A})^{\frac 14}L^{2}_{x}\hookrightarrow H^{\frac 14}L^{2}_{x},
\end{align*}
and hence
\be\label{comp12}
\norm{H^{-\frac 14}(-\lap_{A})^{\frac 14} \phi}_{L^{2}_{x}}\lesssim \norm{(-\lap_{A})^{-\frac 14}(-\lap_{A})^{\frac 14}\phi}_{L^{2}_{x}}=\norm{\phi}_{L^{2}_{x}}.
\ee
For future reference, we remark $\| |\nabla|^{1/2} \phi \|_{L_x^2} \sim \| H^{1/4} \phi \|_{L_x^2} \sim \| (-\lap_A)^{1/4} \phi \|_{L_x^2}$. Next, from the proof of Theorem \ref{compHD} we have
\begin{cor}\label{comparison_cor} With the same assumptions as in Theorem \ref{compHD} we have
\begin{align}
\norm{H^{\frac 12} \phi}_{L^{q}}&\leq C\norm{\nabla\phi}_{L^{q}}, \quad 1<q<n,\label{comp3}\\
\norm{\nabla \phi}_{L^{2}}&\leq C\norm{H^{\frac 12}\phi}_{L^{2}}.\label{comp4}
\end{align}
\end{cor}
\begin{proof}
For \eqref{comp3} interpolate (2.5) and (2.7) in \cite{JFA}.  \eqref{comp4} is (2.12) in \cite{JFA}.
\end{proof}
The homogenous Strichartz estimate was established in \cite{JFA}
\begin{thm}[magnetic Schr\"odinger Strichartz, \cite{JFA}]\label{mss}  Let $n\geq 3$.  If $(A_{0}, A)$ satisfy \eqref{FVc0}-\eqref{FVc2}, \eqref{FVc3}-\eqref{FVc6},
then for any Schr\"odinger admissible pair $(q,r)$, the following Strichartz estimates hold:
\begin{align}
\norm{e^{-itH}\phi}_{L^{q}_{t}L^{r}_{x}}\leq C\norm{\phi}_{L^{2}}, \quad\frac 2q+\frac nr=\frac n2, \ q\geq 2, \ q\neq 2 \ \mbox{if} \ n=3,
\end{align}
and if $n=3$, then at the endpoint we have
\begin{align}
\norm{|\nabla|^{\frac 12}e^{-itH}\phi}_{L^{2}_{t}L^{6}_{x}}\leq \norm{H^{\frac 14}\phi}_{L^{2}}.
\end{align}
\end{thm}
In the proof of the inhomogeneous Strichartz estimate we rely on the Christ-Kiselev Lemma.   
\begin{thm} [Christ-Kiselev Lemma \cite{ChristKiselev} and see \cite{HTW06, SmithSogge, TaoSpherical} ] \label{CK}
Let $X, Y$ be Banach spaces and suppose $$T:  L^p([a,b] ; X) \rightarrow  L^q([a,b] ; Y),$$
where $-\infty\leq a<b\leq \infty$ is an operator given by
\begin{equation*}
Tf(t):= \int_a^t K(t,s) f(s)ds ,
\end{equation*}
for some operator-valued kernel $K(t,s)$ from $X$ to $Y$, and let $T$ satisfy
\begin{align} \label{T_bdd}
\| Tf \|_{L^q([a,b]; Y)} \leq C \| f \|_{L^p([a,b]; X)},
\end{align}
where $1 \leq p < q \leq \infty$ and $C >0$ is independent of $f$.  Now define 
\begin{equation*}
\tilde{T}f (t) = \int_a^b K(t,s) \chi_{(a, t)} (s) f(s) ds = \int_{a}^{t} K(t,s) f(s)ds.
\end{equation*}
Then 
\begin{align*}
\| \tilde{T}f \|_{L^q([a,b]; Y)} \leq 2 \frac{2^{\frac 2q - \frac 2p}}{1 - 2^{\frac 1q - \frac 1p}} C \| f \|_{L^p([a,b]; X)}.
\end{align*} 
\end{thm}
\section{Local Conservation Laws and Virial Identity}\label{cvi}
Recall
\be\tag{mNLS}
\begin{split}
iD_{t}u+ \lap_{A} u&= \mu g(|u|^2) u, \\
u(0)&=u_{0},
\end{split}
\ee
where $\mu \in \R$ and $g$ is a real valued $C^1$ function such that $g(0) = 0$.  For the convenience of the computations we write down an equivalent form of mNLS as
\be\label{imnls}
D_{t}u= i\lap_{A} u-i\mu g(|u|^2) u.
\ee
The virial identity for the linear magnetic Schr\"odinger equations was already established in \cite{FanelliVega09} with a potential $V$ (which is $A_0$ in the above equation) satisfying
\[
\norm{Vu}_{L^{2}_{x}}\leq(1-\epsilon)\norm{\lap_{A}u}_{L^{2}_{x}}+C\norm{u}_{L^{2}_{x}}, \quad \epsilon>0.
\] 
We discuss local conservation laws.
\subsection{Local conservation laws}
Let $G$ be a real valued function such that 
\be\label{Gg}
G'(x) = xg'(x).
\ee
Define pseudo-stress energy tensors as
\begin{align}
T_{00}&=\frac 12\abs{u}^{2}\label{t00}, \\
T_{j0}&=\mathcal Im \{\bar uD_ju\}\label{tj0}, \\
T_{jk}&=2\mathcal Re \{ D_ju \overline{D_ku}\}-\frac{1}{2}\delta_{jk} \lap \abs{u}^{2} + \mu \delta_{jk} G(|u|^2)\label{tjk}, 
\end{align}
for $1 \leq j, k \leq n$.   We have the first local conservation law
\be
\partial_\alpha T_{\alpha0}=0,
\ee
which can be checked easily as follows.
\begin{align*}
\pd_t T_{00} &=\mathcal Re\{\bar uD_{t}u\}\quad\mbox{by} \ \eqref{p1}\\
&=\mathcal Re\{\bar u( i\lap_{A} u-i\mu g(|u|^2) u )\}\quad\mbox{by} \ \eqref{imnls}\\
&=-\mathcal Im\{\bar u\lap_{A} u\}+\mathcal Im\{\mu g(|u|^2) \abs{u}^{2}\} \quad (\mbox{since}\  \mathcal Re\{ iz\}=-\mathcal Im z, \; z\in \mathbb C)\\
&=-\mathcal Im\{\bar u \lap_{A} u\}. 
\end{align*}
Now we compute 
\begin{align*}
\pd_j T_{j0} &= \mathcal Im \{\overline{D_{j}u}D_{j}u\}+\mathcal Im\{\bar u \lap_{A}u \} \quad\mbox{by} \ \eqref{p1}\\
&= \mathcal Im \{ \bar{u}\lap_A u\}.
\end{align*}
Hence $\pd_{\alpha} T_{\alpha 0} = 0$ as needed.

Next, we show we have 
\be\label{tj}
\partial_\alpha T_{j\alpha}= 2F_{\alpha j}T_{\alpha0}.
\ee
To establish \eqref{tj} we compute
\begin{align*}
\partial_0 T_{j0}&=\mathcal Im \{\overline{D_t u}D_j u +  \bar uD_t D_ju\} \quad\mbox{by} \ \eqref{p1}\nonumber\\
 &=\mathcal Im \{\overline{D_t u} D_j u   +\bar{u}iF_{0j}u+\bar{u} D_j D_tu \} \quad\mbox{by} \ \eqref{p2}\nonumber\\
 &=\mathcal Im \{\overline{(i \lap_A u - i\mu g(|u|^2) u} ) D_j u  \}  +  F_{0j}\abs{u}^{2}+\mathcal Im\{ \bar u D_j \left( i \lap_A u - i\mu g(|u|^2 )  u \right) \} \quad\mbox{by} \ \eqref{imnls}\nonumber\\
 &=F_{0j}\abs{u}^{2}-\mathcal Re \{\overline{\lap_{A}u}D_{j}u-\bar u D_{j}(\lap_{A}u)\}+\mathcal Re \{\mu g(\abs{u}^{2})\bar u D_{j}u-\bar u D_{j}(\mu g(\abs{u}^{2})u)\}.
\end{align*}
Since by \eqref{p3}
\begin{align*}
&\mathcal Re \{\mu g(\abs{u}^{2})\bar u D_{j}u-\bar u D_{j}(\mu g(\abs{u}^{2})u)\}\\
&\quad=\mathcal Re \{ \mu g(\abs{u}^{2})\bar u D_{j}u - \bar u \mu \partial_{j}(g(\abs{u}^{2}))u - \bar u \mu g(\abs{u}^{2})D_{j}u\}\\
&\quad=-\mu g'(\abs{u}^{2})\abs{u}^{2}\partial_{j}\abs{u}^{2},
\end{align*}
we have
\begin{align}
\partial_0 T_{j0}=F_{0j}\abs{u}^{2}-\mathcal Re \{\overline{\lap_{A}u}D_{j}u-\bar u D_{j}(\lap_{A}u)\}-\mu g'(\abs{u}^{2})\abs{u}^{2}\partial_{j}\abs{u}^{2}.
\end{align}
Next observe
\begin{align*}
 \lap \abs{u}^{2}&=2 \partial_k\mathcal Re  \{ \bar uD_k u\} \\
&=2 \lvert \nabla_{A} u\rvert^2+ 2\mathcal Re \{\overline{u} \lap_A u \}.
\end{align*}
Hence
\begin{align*}
\partial_k T_{jk}&=2\mathcal Re\{ D_kD_ju\overline{D_ku}+ D_ju\overline{\lap_{A}u}\big\}-\frac{1}{2} \partial_j \lap \abs{u}^{2} + \mu \pd_j G(|u|^2)\\
&=2\mathcal Re\{ D_kD_ju\overline{D_ku}+ D_ju\overline{\lap_{A}u}- \frac 12\partial_j( \bar u\lap_{A}u)\} -\partial_j \lvert \nabla_{A}u\rvert^2 + \mu G'(|u|^2) \pd_j |u|^2
\end{align*}
Now 
\begin{align*}
 \mathcal Re ( D_kD_ju\overline{D_ku})&=  \mathcal Re ( iF_{kj}u\overline{D_{k}u}+D_{j}D_{k}u\overline {D_{k}u})\\
&=  \mathcal Re(iF_{kj}u\overline{D_{k}u})+\frac 12 \partial_{j}\abs{\nabla_{A}u}^{2}.
\end{align*}
 It follows
\begin{align*}
\partial_k T_{jk}&=2\mathcal Re \{iF_{kj}u\overline{D_{k}u}
+D_ju\overline{\lap_{A}u}-\frac{1}{2} \partial_j(\bar u \lap_{A}u)\} + \mu G'(|u|^2) \pd_j |u|^2. \label{c2} 
\end{align*}
Combining and using \eqref{Gg} we have
\begin{align*}
\partial_0 T_{j0}+\partial_k T_{jk}&=
F_{0j}\abs{u}^{2}-\mathcal Re \{\overline{\lap_{A}u}D_{j}u-\bar u D_{j}(\lap_{A}u)\} \\
& \quad  +2\mathcal Re \{iF_{kj}u\overline{D_{k}u}
+D_ju\overline{\lap_{A}u}-\frac{1}{2} \partial_j(\bar u \lap_{A}u)\}.
\end{align*}
Since $\partial_j (\bar u \lap_{A}u) = \overline{D_j u} \lap_A u + \bar u D_j \lap_A u$ from \eqref{p1}, 
\begin{align*}
\mathcal Re \{ -\overline{\lap_A u} D_j u + \bar u D_j \lap_A u + 2 D_j u \overline{\lap_A u} - \partial_j(\bar u \lap_{A}u)\} = 0,
\end{align*}
and 
\begin{align*}
\partial_0 T_{j0}+\partial_k T_{jk} &= F_{0j}\abs{u}^{2}-  2\mathcal Im \{F_{kj}u\overline{D_{k}u}\}\\
&=F_{0j}\abs{u}^{2}+2F_{kj} \mathcal Im\{\bar u D_{k}u\}		\\
&=2F_{\a j}T_{\a0},			\end{align*}
as needed.  We are now ready to proceed to the virial identity.
\subsection{Virial identity for mNLS}
Let $a: \R^{n}\rightarrow \R$.  Define (gauged) Morawetz action by
\be\label{MA}
M_{a}(t)=\int_{\R^{n}}\partial_{j} a T_{j0} dx.
\ee
Note from H\"older's inequality and the definition of $T_{j0}$, we immediately have
 \begin{align*}
\sup_{[0,T]} M_{a}(t)\leq \norm{\nabla a}_{L^{\infty}_{x}}\norm{u}_{L^{2}_{x}}\norm{\nabla_{A}u}_{L^{2}_{x}}.
\end{align*}
This can be refined just like it was in the classical case in \cite{CKSTT04}.  Using \cite[Lemma 3.1]{FanelliVega09} we have (we note the statement of the lemma gives $\norm{H^{\frac 14} u}^{2}_{L^{2}_{x}}$, but the following can be deduced from the proof)
\begin{align} \label{supM}
\sup_{[0,T]} M_{a}(t)\leq C \norm{(-\lap_A)^{\frac 14} u}^{2}_{L^{2}_{x}},
\end{align}
if we assume $\abs{\nabla a}, \abs{x}\lap a$ to be bounded, which they always are in our case.
Next, following \cite{CGT08} we obtain the following lemma.
\begin{lem}[Generalized virial identity]\label{Virial} Let $a: \R^{n}\rightarrow \R$, and $u$ be a solution of (mNLS). Then 
\be\label{virial1}
\begin{split}
&{M_{a}(T)}-{M_{a}(0)}=\\
&\quad\int^{T}_{0}\int_{\R^{n}} \big(2 \partial_{j}\partial_{k}a\mathcal Re (D_{j}u\overline{D_{k}u})-\frac{\lap^{2}a}{2}\abs{u}^{2}  + \mu \lap a  G(|u|^2)
+2\partial_{j}aF_{\alpha j}T_{\alpha 0} \big) dxdt .
\end{split}
\ee
\end{lem}
\begin{proof}
By \eqref{MA}, \eqref{tj} and integration by parts,
\begin{align*} 
\partial_{t} M_{a}(t)&= \int_{\R^{n}}\left(\partial_{k}\partial_{j}a T_{jk}+2\partial_{j}aF_{\alpha j}T_{\alpha 0}\right)dx\\
\label{virial2} &= \int_{\R^{n}}(2\partial_{j}\partial_{k}a\mathcal Re (D_{j}u\overline{D_{k}u})-\frac{\lap^{2}a}{2}\abs{u}^{2} 
+\mu \lap a  G(|u|^2) +2\partial_{j}aF_{\alpha j}T_{\alpha 0})dx.
\end{align*}
\eqref{virial1} now follows by the fundamental theorem of calculus.
\end{proof}
\begin{cor}\label{vcor}  If $a$ is convex and $\mu G(\abs{u}^{2})\geq 0$ we can further conclude
\be\label{virial3}
\int^{T}_{0}\int_{\R^{n}}2\partial_{j}aF_{\alpha j}T_{\alpha 0}-\frac{\lap^{2}a}{2}\abs{u}^{2}\ dxdt\lesssim  \sup_{[0,T]}\abs{M_{a}(t)}.
\ee
\end{cor}
\begin{proof}
This is easy to see since if $a$ is convex, we can first show that
\be\label{convex1}
\partial_{j}\partial_{k}a\mathcal Re (D_{j}u\overline{D_{k}u})\geq 0.
\ee
Indeed,  we know if a function $a: \R^{n}\rightarrow \R$ is convex then for $X \in \R^{n}$,
\be\label{convex}
\partial_{j}\partial_{k}aX^{j}X^{k}\geq 0.
\ee
We apply this twice to conclude \eqref{convex1}. 
Define vectors $X$, $Y$ by 
\begin{align*}
X^{i}&=\mathcal Re D_{i}u\quad \mbox{for}\quad 1\leq i \leq n,\\
Y^{i}&=\mathcal Im D_{i}u\quad \mbox{for}\quad 1\leq i \leq n.
\end{align*}
Next since for general $z,w \in \mathbb{C}$,
\[
\mathcal Re(z\bar w)=\mathcal Re{z}\mathcal Re{w}+\mathcal Im{z}\mathcal Im{w},
\]
we have
\[
\partial_{j}\partial_{k}a\mathcal Re (D_{j}u\overline{D_{k}u})=\pd_{j}\pd_{k}aX^{j}X^{k}+\pd_{j}\pd_{k}aY^{j}Y^{k}\geq 0,
\]
by \eqref{convex}.
Finally since $a$ is convex and the Hessian, $(H_{jk})=(\partial_{j}\partial_{k}a)$ is positive-semidefinite, the trace, $\tr(H_{jk})=\lap a\geq 0$, which implies
\[
\mu \int_{\R^n} \lap aG(|u|^2) dx\geq 0,
\] 
and the result follows.

\end{proof}
We end this section by a brief discussion of the conservation of mass and energy for the mNLS.  From \cite{FanelliVega09} we have
\[
\norm{e^{-itH}\phi}_{\mathcal {\dot H}^{s}}=\norm{f}_{{\mathcal {\dot H}^{s}}} \quad s\geq 0,
\]
where $\norm{f}_{\mathcal {\dot H}^{s}}=\norm{H^{\frac s2}f}_{L^{2}}$.
This in particular implies conservation of mass and energy for the linear magnetic Schr\"{o}dinger equations.  In case of mNLS we have
\begin{lem}[Conservation of mass and energy]
Let $H=-\Delta_{A}+A_{0}$ be self-adjoint and positive on $L^{2}$, $F'=g$ and let $u$ solve mNLS.  Then for every $t>0$
\begin{align}
\norm{u(t)}_{L^{2}}&=\norm{u_{0}}_{L^{2}}\label{mass},\\
\int_{\R^{n}} \abs{H^{\frac 12}u(t)}^{2} dx +\mu F(\abs{u}^{2}) dx &= \int_{\R^{n}}\abs{H^{\frac 12}u(0)}^{2} dx +\mu F(\abs{u(0)}^{2}) dx\label{energy}.
\end{align}
\end{lem}
\begin{proof}
\eqref{mass} follows by integrating in space $\partial_{t}T_{00}+\partial_{j}T_{j0}=0$, and \eqref{energy} by a direct computation using the equation.
\end{proof}

\section{Interaction Morawetz Estimates}\label{interaction}
 
As in \cite{CGT08} we use the following notation 
\[
\rho =T_{00},\quad p_{j}=T_{j0},
\]
and
\[
T_{jk}=\sigma_{jk}-\delta_{jk}\Delta \rho+\mu \delta_{jk}G(2\rho),
\]
where $\sigma_{jk}=\frac 1\rho(p_{j}p_{k}+\partial_{j}\rho\partial_{k}\rho)=2\mathcal Re(D_{j}u\overline{ D_{k}u})$.
Then we rewrite the local conservation laws as
\begin{align}
&\partial_{t}\rho+\partial_{j}p_{j}=0, \label{newt0}\\
&\partial_{t}p_{j}+\pd_{k}(\sigma_{jk}- \delta_{jk}\Delta \rho+\mu \delta_{jk}G(2\rho))=2F_{\alpha j}T_{\alpha 0} \label{newtj}.
\end{align}
\subsection{Proof of Theorem \ref{thm1} using the commutator vector operators}
The Morawetz action \eqref{MA} for a \emph{tensor product} of two solutions $u_{1}=u_{2}=u$ with $a=\abs{x-y}$ can be rewritten as
\begin{align*}
M(t)&=\int_{\R^{n}\otimes \R^{n}} \partial_{j}aT_{j0}dx dy\\
       &=\int_{\R^{n}\otimes \R^{n}} \frac{x-y}{\abs{x-y}}\cdot\mathcal Im\{\bar u(t,x)\nabla_{A}u(t,x)\}\abs{u(t,y)}^{2}dx dy\\
       &\quad -\int_{\R^{n}\otimes \R^{n}} \frac{x-y}{\abs{x-y}}\cdot\mathcal Im\{\bar u(t,y)\nabla_{A}u(t,y)\}\abs{u(t,x)}^{2}dx dy\\
&=2\int_{\R^{n}\otimes \R^{n}}\frac{x-y}{\abs{x-y}}\cdot \left(\vec{p}(t, x)\rho(t,y)-\vec{p}(t,y)\rho(t,x) \right)dxdy\\
&=4\int_{\R^{n}\otimes \R^{n}}\frac{x-y}{\abs{x-y}}\cdot \vec{p}(t, x)\rho(t,y)dxdy.
\end{align*}
Following \cite{CGT08} we use operators $|\nabla|^{-(n-1)}$ and $\vec X$ defined by
\[
|\nabla|^{-(n-1)}f(x)=\int_{\R^n}\frac{1}{|x-y|}f(y)dy,\quad \vec{X}=[x;|\nabla|^{-(n-1)}],
\]
so
\begin{align}
\vec X f(x)&=\int_{\R^{n}}\frac{x-y}{\abs{x-y}}f(y) dy,\\
\ip{\vec F\ |\ \vec X g}&=\int_{\R^{n}}\vec F(x)\cdot \vec X g(x)dx=-\ip{\vec X\cdot \vec F\ | \ g}\label{asX}.
\end{align}
Further, a computation shows 
\[
(\partial_{j}X^{k})f(x)=\int_{\R^n}\eta_{kj}(x,y)f(y)dy,
\]
where
\[
\eta_{kj}(x,y)=\frac{\delta_{kj}|x-y|^2-(x_j-y_j)(x_k-y_k)}{|x-y|^3},
\]
and
\[
\partial_{j}X^{j}=n|\nabla|^{-(n-1)}+[x_j ; R_{j}]=(n-1)|\nabla|^{-(n-1)},
\]
where $R_j = \partial_j | \nabla|^{-(n-1)}$. The crucial observation made in \cite{CGT08} was that the derivatives of  $\vec{X}$ are positive definite.  Using the above operators we write
 \[
 M(t)=4\ip{ [x;|\nabla|^{-(n-1)}]\rho(t) \  | \ \vec{p}(t)}=  4\ip{\vec{X}\rho(t) \  | \ \vec{p}(t)}.
 \]
Then 
\[
\partial_{t}M(t)=4\ip{ \vec{X}\partial_{t}\rho(t) \  | \ \vec{p}(t)}+4\ip{ \vec{X}\rho(t) \  | \  \partial_{t}\vec{p}(t)}=I + II.
\]
By \eqref{asX}, and \eqref{newt0}
\begin{align*}
I=-4\ip{\partial_{t}\rho(t) \  | \  \vec{X}\cdot \vec{p}(t)} = 4 \ip{ \partial_{j}p_{j}(t) \   | \  \vec{X}\cdot \vec{p}(t)}= -4 \ip{ p_{j}(t) \   | \  \partial_{j} X^{k} p_{k}(t)}.
\end{align*}
And by  \eqref{newtj} 
\begin{align*}
II&=4\ip{\pd_{k} X^{j}\rho(t) \  | \ \sigma_{jk}- \delta_{jk}\Delta \rho+\mu \delta_{jk}G(\rho)}+4\ip {X^{j}\rho(t) \ | \ 2 F_{\alpha j}T_{\alpha 0}}\\
&=4\ip{\pd_{k} X^{j}\rho(t) \  | \  \frac {1}{\rho}(p_{j}p_{k}+\partial_{j}\rho\partial_{k}\rho)-\delta_{jk}\Delta \rho+\mu \delta_{jk}G(2\rho)}+4\ip {X^{j}\rho(t) \ | \ 2F_{\alpha j}T_{\alpha 0}}.
\end{align*}
It follows
\[
\partial_{t}M(t)=P_{1}+P_{2}+P_{3}+P_{4}+P_{5},
\]
where 
\begin{align*}
P_{1}&=4\ip{ \frac{1}{\rho}\partial_{j}\rho \partial_{k}\rho \  | \ \partial_{k}X^{j}\rho(t)},\\
P_{2}&=4\ip{ \frac{1}{\rho} p_jp_k \  | \ \partial_{k}X^{j}\rho(t)}-4\ip{ p_j \  | \ \partial_{j}X^{k}p_k(t)},\\
P_{3}&=4\ip{ (-\Delta \rho) \  | \ \partial_{j}X^{j}\rho(t)},\\
P_{4}&=4\ip{ \mu G(2\rho) \  | \ \partial_{j}X^{j}\rho(t)},\\
P_{5}&=8\ip {X^{j}\rho(t) \ | \ F_{\alpha j}T_{\alpha 0}}.
\end{align*}
We discuss the positivity of each term.  This analysis is also the same as in \cite{CGT08}, but the difference is that the momentum vector $\vec p$ is now covariant, and we also have to address $P_{5}$.  We briefly sketch the main ideas for $P_{1}$ through $P_{4}$ for completeness (for details see \cite{CGT08}).

Since $\partial_{j}X^{k}$ is positive definite, $P_{1}\geq 0$.   For $P_{2}$ define the two point momentum vector
\[
\vec J (x,y)=\sqrt{\frac{\rho(y)}{\rho(x)}}\vec p (x)-\sqrt{\frac{\rho(x)}{\rho(y)}}\vec p (y).
\]
Then (see \cite{CGT08} for details)
\[
P_{2}=2\ip{J^{j}J_{k} \  | \ \partial_{j}X^{k}} \geq 0,
\]
since again $\partial_{j}X^{k}$ is positive definite.  For $P_3$ using $-\Delta=|\nabla|^2$ ,
\[
P_{3}=4(n-1)\ip{ (|\nabla|^2\rho)(t) \  | \ |\nabla|^{-(n-1)}\rho(t)}=(n-1)\||\nabla|^{-\frac{n-3}{2}}(|u|^2)\|_{L^2}^2,
\]
and
\[
P_{4}=4\ip{ \mu G(2\rho) \  | \ (\partial_{j}X^{j})\rho(t)}=4(n-1)\ip{ \mu G(2\rho)\  | \ |\nabla|^{-(n-1)}\rho(t)} \geq0
\]
as long as $\mu G(2\rho)\geq 0$.  

Now, integrating in time we have
\[
 \int_0^T P_3 dt +\int_0^T P_5 dt \leq M(T) - M(0),
\]
so the estimate follows by \eqref{supM} \emph{if} we can handle the last term $P_{5}$.

We cannot expect $P_{5}$ to be positive (see the appendix).  Examples when $B_{\tau}=0$ were given in  \cite{FanelliVega09} (note this still leaves the term involving $F_{0j}$). In general, as shown below, we can control $P_{5}$ by imposing the conditions \eqref{FVc0}-\eqref{FVc2} as they allow us to take advantage of the smoothing estimates proved in \cite{FanelliVega09}.  In addition, we also require \eqref{latest}-\eqref{latestc4}.
\subsection{$P_5:$ Replacement of positivity condition by bounds on $F$}
Suppose \eqref{FVc0}-\eqref{FVc2} hold.  Then
\begin{align*}
\int^{T}_{0}P_5dt&= 4\int^{T}_{0}\int_{\R^{2n}} \frac{x_j - y_j}{|x-y|}F_{k j}(x)p_{k} (x)  \abs{u(y)}^2 dxdydt\\
&\quad+ 2\int^{T}_{0}\int_{\R^{2n}} \frac{x_j - y_j}{|x-y|}F_{0 j}(x)\abs{u(x)}^{2}  \abs{u(y)}^2 dxdydt\\
&=I + II.
\end{align*}
\subsubsection{Estimates for $n=3$.}
Choose $0<b<1$.  Impose \eqref{latest}- \eqref{latestc2}.  Since $\vec p=\mathcal Im\{\bar u \nabla_{A}u \}$ we get
\begin{align*}
I&\lesssim \int^{T}_{0}\int_{\R^{6}} \abs{dA(x)}\abs{u(x)}\abs{\nabla_{A}u (x)}  \abs{u(y)}^2 dxdydt\\
&=\norm{u_0}^{2}_{L^{2}_{y}}\int^{T}_{0} \int_{\R^{3}} \abs{dA(x)}\abs{u(x)}\abs{\nabla_{A}u (x)} dx dt\\
&\lesssim \norm{u_0}^{2}_{L^{2}_{y}}\int^{T}_{0} \int_{\R^{3}} \abs{dA(x)}^{2-2b}\abs{\nabla_{A}u (x)}^2 dx dt+
\norm{u_0}^{2}_{L^{2}_{y}}\int^{T}_{0} \int_{\R^{3}} \abs{dA(x)}^{2b}\abs{u (x)}^2 dx dt \\
&=Ia+Ib.
\end{align*}
Next
\begin{align*}
Ia&=\norm{u_{0}}^{2}_{L^{2}}\int_0^T\sum_{j \in \mathbb Z} \int_{C_j} \abs{dA(x)}^{2-2b}\abs{\nabla_{A}u (x)}^2 dx dt\\
&\leq\norm{u_{0}}^{2}_{L^{2}}\sum_{j \in \mathbb Z}\sup_{x \in C_{j}}2^{j+1}\abs{dA(x)}^{2-2b}\int_0^T \int_{C_j} \frac{\abs{\nabla_{A}u (x)}^2}{2^{j+1}} dx dt\\
&\leq\norm{u_{0}}^{2}_{L^{2}}\sum_{j \in \mathbb Z}\sup_{x \in C_{j}}2^{j+1}\abs{dA(x)}^{2-2b}\left(\sup_{R}\int_0^T \int_{\abs{x}\leq R} \frac{\abs{\nabla_{A}u (x)}^2}{R} dx dt\right)\\
&\leq C\norm{u_0}^{2}_{L^{2}}\sup_{t\in [0,T]}\norm{(-\lap_A)^{\frac 14} u(t)}^{2}_{L_x^{2}},
\end{align*}
by \eqref{ey3} and \eqref{latest}.
\begin{align*}
Ib&= \norm{u_{0}}^{2}_{L^{2}}\int^{T}_{0}\int^{\infty}_{0}\int_{\abs{x}=R} R^{2}\abs{dA(x)}^{2b}\frac{\abs{u(x)}^{2}}{R^{2}}d\sigma dRdt \\
 &\leq \norm{u_{0}}^{2}_{L^{2}}\left( \int^{\infty}_{0}\sup_{\abs{x}=R} \abs{x}^{2}\abs{dA(x)}^{2b}dR\right) \left(\int^{T}_{0}\sup_{R>0}\int_{\abs{x}=R} \frac{\abs{u(x)}^{2}}{R^{2}}d\sigma dt\right) \\
 &\leq C \norm{u_{0}}^{2}_{L^{2}}\sup_{t\in[0,T]}\norm{(-\lap_A)^{\frac 14} u(t)}^{2}_{L_x^{2}} ,
\end{align*} 
 by \eqref{latestc1} and \eqref{ey3}.  To estimate $II$ note that $A$ is independent in time and $F_{0j}=-\partial_{j}A_{0}$. Then
\begin{align*}
II&\lesssim \int^{T}_{0}\int_{\R^{6}} \abs{\nabla A_{0}(x)}\abs{u(x)}^{2}  \abs{u(y)}^2 dxdydt\\
&=\norm{u_{0}}^{2}_{L^{2}}\int^{T}_{0}\int^{\infty}_{0}\int_{\abs{x}=r}  \abs{ \nabla A_{0}(x)}\abs{u(x)}^{2} d\sigma dr dt \\
&=\norm{u_{0}}^{2}_{L^{2}} \int^{T}_{0}\int^{\infty}_{0}\int_{\abs{x}=r}  \abs{x}^{2}\abs{\nabla A_{0}(x)}\frac{\abs{u(x)}^{2}}{\abs{x}^{2}} d\sigma dr dt \\
&\leq \norm{u_{0}}^{2}_{L^{2}} \norm{\abs{x}^{2}\nabla A_{0}(x)}_{L^{1}_{r}L^{\infty}(S_{r})} \sup_{r >0} \int^{T}_{0} \int_{\abs{x}=r} \frac{\abs{u(x)}^{2}}{\abs{x}^{2}} d\sigma dt \\
&\leq C\norm{u_{0}}^{2}_{L^{2}}\sup_{t\in[0,T]}\norm{(-\lap_A)^{\frac 14} u(t)}^{2}_{L_x^{2}},
\end{align*}
by \eqref{ey3} and \eqref{latestc2}.
The estimates for $n\geq 4$ are analogous. 
\subsubsection{Estimates for $n\geq 4$.}
Just as before, we write
\begin{align*}
I\leq Ia +Ib,
\end{align*}
where $Ia$ is estimated using \eqref{latest} and \eqref{ey4}.  For $Ib$ we have 
\begin{align*}
Ib\leq& \norm{u_{0}}^{2}_{L^{2}}\big( \sup_{\abs{x}} \abs{x}^{3}\abs{dA(x)}^{2b} \big) \left( \int^{T}_{0}\int_{\R^{n}} \frac{\abs{u(x)}^{2}}{\abs{x}^{3}}dxdt\right) \\
 \leq& C \norm{u_{0}}^{2}_{L^{2}}\sup_{t\in[0,T]}\norm{(-\lap_A)^{\frac 14} u(t)}^{2}_{L^{2}_{x}} ,
\end{align*} 
 by \eqref{latestc3} and \eqref{ey4}.
Next,
\begin{align*}
II&\lesssim  \norm{u_{0}}^{2}_{L^{2}}\int^{T}_{0} \int_{\R^{n}}  \abs{ \nabla A_{0}(x)}\abs{u(x)}^{2} dxdt \\
&=\norm{u_{0}}^{2}_{L^{2}}\int^{T}_{0}\int_{\R^{n}}  \abs{x}^{3}\abs{\nabla A_{0}(x)}\frac{\abs{u(x)}^{2}}{\abs{x}^{3}} dx dt \\
&\leq C\norm{u_{0}}^{2}_{L^{2}}\sup_{t\in[0,T]}\norm{(-\lap_A)^{\frac 14} u(t)}^{2}_{L^{2}_{x}},
\end{align*}
by \eqref{latestc4} and \eqref{ey4}.
 
\section{Proof of the Inhomogeneous Strichartz Estimate, Theorem \ref{thm2}}\label{nh}

Let $N(t,x)$, $t \geq 0$ be a space-time function which is sufficiently regular and $u$ be the solution of \eqref{inhomog_eq}. Note that 
\begin{align*}
u(t) = \int_0^t e^{-iH(t-s)} N(s, \cdot) ds =: \int_0^t K(t,s) N(s, \cdot) ds =: \tilde T N ,
\end{align*}
by Duhamel's principle and define $Tf = \int_0^\infty K(t,s) N(s, \cdot) ds$. By the Christ--Kiselev lemma, 
\begin{align*}
\| u \|_{L_t^q L_x^r} = \| \tilde T N \|_{L_t^q L_x^r} \leq c\| T \|_{L_t^{\tilde q '} L_x^{\tilde r '} \rightarrow  L_t^q L_x^r} \| N \|_{L_t^{\tilde q '} L_x^{\tilde r '}}. 
\end{align*}
So it is enough to show
\begin{align*}
\| T g \|_{L_t^q L_x^r} \leq C \| g \|_{L_t^{\tilde q '} L_x^{\tilde r '}}, 
\end{align*}
for any $g \in L_t^{\tilde q '} L_x^{\tilde r '}$, $\tilde q'<q$. From the definition of $Tg$, Strichartz estimate and self-adjointness of $H$, 
\begin{align*}
\| Tg \|_{L_t^q L_x^r} &= \left \| e^{-itH} \int_0^\infty e^{isH} g(s, \cdot) ds \right \|_{L_t^q L_x^r} \\
&\leq C \left \| \int_0^\infty e^{isH} g(s, \cdot) ds \right \|_{L_x^2} \\
&=C \sup_{\| \phi \|_{L_x^2} = 1} \langle \phi, \int_0^\infty e^{isH} g(s, \cdot) ds  \rangle \\
&= C\sup_{\| \phi \|_{L_x^2} = 1} \int_0^\infty  \langle e^{-isH} \phi,  g(s, \cdot) \rangle  ds  \\
&\leq C\| e^{-isH} \phi \|_{L_t^{\tilde q} L_x^{\tilde r}} \| g \|_{L_t^{\tilde q '} L_x^{\tilde r '}} \\
&\leq C \| g \|_{L_t^{\tilde q '} L_x^{\tilde r '}} ,
\end{align*}
and proof is completed.

\section{Application to Magnetic Nonlinear Schr\"{o}dinger equations}\label{app}

In this section, we show applications of previous estimates to global existence and scattering. For simplicity, we consider magnetic NLS with defocusing cubic nonlinearity in $\R^{1+3}$. 
\be \label{cubic_eq}
\begin{split}
iu_t- H u &= |u|^2 u, \\
u(0) = u_{0} &\in H^{1}(\R^{3}). 
\end{split}
\ee
We note that by virtue of \eqref{comp3} and \eqref{comp4} we have 
\be\label{h1A}
\norm{u}_{H^{1}(\R^{3})}\sim \norm{u}_{L^{2}(\R^{3})}+\|H^{\frac 12}u\|_{L^{2}(\R^{3})}.
\ee
To establish Theorem \ref{thm3} we begin with a local theory (see \cite{CazenaveEsteban88, DeBouard91, NakamuraShimomura05, Laurent08} for related works).
As mentioned in the introduction, the arguments below resemble what is usually done for the critical NLS (see for example \cite{KMnls}).

\subsection{Local existence}\label{LWP}  
Let $Q=\norm{u_{0}}_{H^{1}(\R^{3})}$. Let $\delta>0$ and suppose
\be\label{lwpe1}
\norm{e^{-itH}u_{0}}_{L^{3}_{t}\dot W^{\frac 12, \frac {18} {5}}_x(\R^{3})}<\delta.
\ee
We show that if $\delta$ is small enough we obtain local existence.  We note that we do not require small data. Also, for any $\delta>0$, we can always assume \eqref{lwpe1} if the time interval is small enough.  To see that,
by \eqref{comp2}, Theorem \ref{mss}, and \eqref{comp1}, we have
\be\label{homog}
\norm{e^{-itH}u_{0}}_{L^{3}_{t}\dot W^{\frac 12, \frac {18} {5}}_x}\leq C \norm{H^\frac 14 (e^{-itH}u_{0})}_{L^{3}_{t}L^{\frac {18} {5}}_x}\leq C\norm{H^\frac14 u_0}_{L^2}\leq CQ,
\ee
so $\norm{e^{-itH}u_{0}}_{L^{3}_{t}\dot W^{\frac 12, \frac {18} {5}}_x}$ is finite.  Hence the time interval can be shrunk enough to make \eqref{lwpe1} hold.
 
We construct a unique solution in the space
\[
X_{a,b}:= \{ u \in C_t^0 ([0, T_0]; H_x^1) \cap L_{[0, T_0]}^3 \dot{W}^{\frac 12, \frac{18}{5}}_x\ \ : \ \  \norm{u}_{L^{3}_{t}\dot W^{\frac 12, \frac {18} {5}}_x}\leq a,\   \norm{u}_{L^{\infty}_{[0,T_{0}]}H^{1}_{x}} \leq b \} ,
\]
where $a=2\delta, b=4CQ$, and $C$ is the maximum of the constant $1$, and the constants $C$ that appear in the estimates below.
 Define the sequence of Picard iterates by
 \[
 u^{0}(t)=e^{-itH}u_{0}\qand u^{k+1}(t)=\Phi(u^{k})(t),\ k\geq 0,
 \]
 where
\begin{align*}
\Phi (u) (t) = e^{-itH} u_0 -i \int_0^t e^{-i(t-s)H} |u(s)|^2 u(s) ds.
\end{align*}
By \eqref{lwpe1} and Theorem \ref{mss}
\[
\norm{u^{0}}_{L^{3}_{t}\dot W^{\frac 12, \frac {18} {5}}_x}\leq \frac a2 \qand \norm{u^{0}}_{L^{\infty}_{[0,T_{0}]}L^{2}_{x}} \leq \frac b4,
\]
and by \eqref{comp4}, Theorem \ref{mss} and \eqref{comp3},
\[
\norm{u^{0}}_{L^{\infty}_{[0,T_{0}]}\dot H^{1}_{x}} \leq \frac b4.
\]
Now suppose that for $k\geq 0, u^{k}\in X_{a,b}.$  Then by Theorem \ref{mss}, Corollary \ref{comparison_cor},  Theorem \ref{thm2} and Sobolev embedding,
\begin{align*}
\| u^{k+1}\|_{L_t^{\infty} \dot H^{1}_x} &\leq \| \nabla e^{-itH} u_0 \|_{L_t^{\infty} L_x^2} + \| \nabla \int_0^t e^{-i(t-s)H} |u^{k}(s)|^2 u^{k}(s) ds \|_{L_t^{\infty} L_x^2} \\
&\leq C\| H^{\frac 12} e^{-itH} u_0  \|_{L_t^{\infty} L_x^2} + C\| H^{\frac 12} \int_0^t e^{-i(t-s)H} |u^{k}(s)|^2 u^{k}(s) ds \|_{L_t^{\infty} L_x^2} \\
&\leq C\| H^{\frac 12} u_0 \|_{L_x^2} + C\|  H^{\frac 12} (|u^{k}|^2 u^{k}) \|_{L_t^{\frac 32} L_x^{\frac {18}{13}}} \\
&\leq C \| \nabla u_0 \|_{L_x^2}  +  C\|  \nabla(|u^{k}|^2 u^{k}) \|_{L_t^{\frac 32} L_x^{\frac {18}{13}}} \\
&\leq CQ+  C\| \nabla u^{k} \|_{L_t^{\infty} L_x^2} \| u^{k} \|_{L_t^3 L_x^9}^2 \\
&\leq \frac b4 +  Cb  \|  u^{k} \|_{L_t^3 \dot W^{\frac 12,\frac {18}5}_{x}}^2 \\
&\leq \frac b4 +  Cb a^{2}.
\end{align*}
Hence, if $a$ is small enough, i.e.,
\be\label{deltaa}
a^{2}\leq \frac {1}{4C},
\ee
\[
\| u^{k+1}\|_{L_t^{\infty} \dot H^{1}_x} \leq b.
\]
Similarly
\begin{align*}
\|u^{k+1} \|_{L_t^{\infty} L_x^2} &\leq \| e^{-itH} u_0 \|_{L_t^{\infty} L_x^2} + \| \int_0^t e^{-i(t-s)H} |u^k(s)|^2 u^k(s) ds \|_{L_t^{\infty} L_x^2} \\
&\leq C \| u_0 \|_{L_x^2} + C \| |u^k|^2 u^k \|_{L_t^1 L_x^2} \\
&\leq C Q+ C \| u^k \|_{L_t^3 L_x^6}^3 \\
&\leq  \frac b4 + C T \| u^k \|_{L_t^{\infty} L_x^6}^3 \\
&\leq\frac b4 + C T \| \nabla u^k \|_{L_t^{\infty} L_x^2}^3\\
&\leq \frac b4 + C T b^3.
\end{align*}
Let 
\be\label{Tcb}
T \leq \frac {1}{4Cb^2}.
\ee
Then
\[
\| u^{k+1}\|_{L_t^{\infty} L^2_x} \leq b.
\]
Next,
\begin{align*}
\| |\nabla|^{\frac 12}u^{k+1} \|_{L_t^{3} L_x^{\frac {18}{5}}} &\leq \| |\nabla|^{\frac 12} e^{-itH} u_0 \|_{L_t^{3} L_x^{\frac {18}{5}}} + \| |\nabla|^{\frac 12} \int_0^t e^{-i(t-s)H} |u^{k}(s)|^2 u^{k}(s) ds \|_{L_t^{3} L_x^{\frac {18}{5}}} \\
&\leq \frac a2+ C\| H^{\frac 14} \int_0^t e^{-i(t-s)H} |u^{k}(s)|^2 u^{k}(s) ds \|_{L_t^{3} L_x^{\frac {18}{5}}} \\
&\leq\frac a2 + C\|  H^{\frac 14} (|u^{k}|^2 u^{k}) \|_{L_t^{\frac {12}{11}} L_x^{\frac 95}} \\
&\leq \frac a2+ C\|  |\nabla|^{\frac 12} (|u^{k}|^2 u^{k}) \|_{L_t^{\frac {12}{11}} L_x^{\frac 95}} \\
&\leq\frac a2+ CT^{\frac 14} \| |\nabla|^{\frac 12} u^{k} \|_{L_t^3 L_x^{\frac {18}{5}}} \| u^{k} \|_{L_t^3 L_x^9} \| u^{k} \|_{L_t^{\infty} L_x^6} \\
&\leq\frac a2 + CT^{\frac 14} \| |\nabla|^{\frac 12} u^{k} \|_{L_t^3 L_x^{\frac {18}{5}}}^2 \| \nabla u^{k} \|_{L_t^{\infty} L_x^2} \\
&\leq\frac a2 + CT^{\frac 14}a^2 b.
\end{align*}
If we require
\be\label{Tc}
T^{\frac 14}\leq \frac {1}{2Cab},
\ee
then
\[
\| |\nabla|^{\frac 12} \Phi (u) \|_{L_t^{3} L_x^{\frac {18}{5}}} \leq a,
\]
which shows the sequence $u^{k}$ belongs to $X_{a,b}$.  To show the sequence converges, we need to consider the differences.  The estimates are similar, and we only show some of the details.

Let 
\[
F(u)=\abs{u}^{2}u,
\]
then we can write
\[
F(u)-F(v)= (u-v) \int_0^1 F_{z} \big( \lambda u + (1-\lambda) v\big)d\lambda+\overline{(u-v)} \int_0^1 F_{\bar z}\big( \lambda u + (1-\lambda) v\big) d\lambda.
\]
Now consider
 \begin{align*}
 \| |\nabla|^{\frac 12}\big(\Phi(u)-\Phi(v)\big) \|_{L_t^{3} L_x^{\frac {18}{5}}} &\leq  C\|  |\nabla|^{\frac 12} \big(F(u)-F(v)\big) \|_{L_t^{\frac {12}{11}} L_x^{\frac 95}}  \\
 &\leq  C\|  |\nabla|^{\frac 12} \big( (u-v) \int_0^1 F_{z} \big( \lambda u + (1-\lambda) v\big)d\lambda        \big) \|_{L_t^{\frac {12}{11}} L_x^{\frac 95}} \\
 &\quad + C\|  |\nabla|^{\frac 12} \big(\overline{(u-v)} \int_0^1 F_{\bar z}\big( \lambda u + (1-\lambda) v\big) d\lambda\big) \|_{L_t^{\frac {12}{11}} L_x^{\frac 95}} \\
 &= I + II.
 \end{align*}
 \begin{align*}
 I&\leq C \sup_{\lambda \in [0,1]}  CT^{\frac 14}\|  |\nabla|^{\frac 12} (u-v) \|_{L_t^{3} L_x^{\frac {18}{5}}}\|  F_{z} \big( \lambda u + (1-\lambda) v\big) \|_{L_t^{3} L_x^{\frac {18}{5}}}\\
  &\quad +C \sup_{\lambda \in [0,1]}  CT^{\frac 14}\|  u-v \|_{L_t^{3} L_x^{9}}\|  |\nabla|^{\frac 12}  F_{z} \big( \lambda u + (1-\lambda) v\big) \|_{L_t^{3} L_x^{\frac {9}{4}}}\\
  &\leq  CT^{\frac 14}ab \|  |\nabla|^{\frac 12} (u-v) \|_{L_t^{3} L_x^{\frac {18}{5}}}\\
  &\quad +C \sup_{\lambda \in [0,1]}  CT^{\frac 14}\|  |\nabla|^{\frac 12} (u-v) \|_{L_t^{3} L_x^{\frac {18}{5}}}\|  |\nabla|^{\frac 12}  \big( \lambda u + (1-\lambda) v\big) \|_{L_t^{3} L_x^{\frac {18}{5}}}   \|   \lambda u + (1-\lambda) v \|_{L_t^{\infty} L_x^{6}}\\
  &\leq  CT^{\frac 14}ab \|  |\nabla|^{\frac 12} (u-v) \|_{L_t^{3} L_x^{\frac {18}{5}}}\\
  &\leq \frac 12\|  |\nabla|^{\frac 12} (u-v) \|_{X_{a,b}},
   \end{align*}
 by \eqref{Tc}.
 We obtain the same bounds for term $II$, and for the other norms in $X_{a,b}$, which show the sequence of the iterates is Cauchy and hence it converges as needed.
\subsection{Global existence}
Let $u$ be the solution of \eqref{cubic_eq} obtained from local existence on time interval $[0,T_*)$.  Suppose $T_*<\infty$.  We show this leads to a contradiction by showing we can extend the solution.    

First, by \eqref{h1A}, \eqref{mass} and \eqref{energy}, the $H^{1}$ norm of $u(t)$ is uniformly bounded.   So by the local well-posedness argument to extend the solution, it is enough to show the existence of $\epsilon_1$ and $\epsilon_2$
such that
\begin{align}\label{lwpc}
\| e^{-i(t-(T_\ast-\epsilon_1))H} u(T_\ast-\epsilon_1) \|_{L_{[T_\ast-\epsilon_1, T_*+\epsilon_2]}^3 \dot{W}_x^{\frac 12, \frac {18}{5}}} <\delta \end{align}
where $\delta$ is specified by \eqref{deltaa} (and $\delta=\frac a2$) and, 
\be\label{eps12}
\epsilon_1+\epsilon_2\leq \frac 1{4Cb^2},\quad (\epsilon_1+\epsilon_2)^\frac 14\leq \frac{1}{2Cab}
\ee
 due to \eqref{Tcb} and \eqref{Tc} respectively.  But similarly as in \eqref{homog}, we have
\be\label{step2}
\norm{e^{-i(t-\tau)H}u(\tau)}_{L^{3}_{[0,\infty)}\dot W^{\frac 12, \frac {18} {5}}_x}\leq  C\norm{u(\tau)}_{H^1}=CQ<\infty,
\ee
for all $\tau \in [0,T_\ast)$.  So we can find $\epsilon_1$ small enough such that
\[
\| e^{-i(t-(T_\ast-\epsilon_1))H} u(T_\ast-\epsilon_1) \|_{L_{[T_\ast-\epsilon_1, T_*]}^3 \dot{W}_x^{\frac 12, \frac {18}{5}}} <\frac \delta 2.
\]
Then by \eqref{step2} and the continuity of the integral, we can find $\epsilon_2$ such that \eqref{lwpc} holds together with \eqref{eps12} as needed (by taking $\epsilon_1$ smaller if necessary).

\subsection{Scattering}
In this section we consider the question of scattering (asymptotic completeness).  
We take a point of view analogous to the classical NLS.  Hence we set out to show that given a solution of the nonlinear mNLS, $u$, there exists a solution of the linear mNLS, $e^{-itH}u_{+}$, such that the $H^{1}_{x}$ norm of the difference of the two solutions goes to $0$ as $t \rightarrow \infty$ (note, due to \eqref{h1A} this also gives convergence of $\norm{H^{\frac 12}(u-e^{-itH}u_{+})}_{L^{2}_{x}}$).  
\newline\indent
Now, following the classical NLS setup for scattering, let $u$ be the solution to the cubic defocusing mNLS with initial data $u_{0} \in H^{1}_{x}$.  We define 
\[
u_{+}=u_{0}-i\int_0^{\infty} e^{isH} |u(s)|^2 u(s) ds.
\]
The convergence in $H^{1}_{x}$ of the difference of $u$ and $e^{-itH}u_{+}$ is then immediate if we can show 
\begin{align*}
\int_0^{\infty} e^{isH} |u(s)|^2 u(s) ds,
\end{align*}
converges in $H^1$. 
Therefore, equivalently, we need to show
\begin{equation}
\label{scattering} \lim_{t \rightarrow \infty} \| \int_t^{\infty} e^{isH} |u(s)|^2 u(s) ds \|_{H^1} = 0.
\end{equation}
We prove (\ref{scattering}) for $L^2$ and $\dot{H}^1$ separately. For $L^2$, we need to show
\begin{align*}
\sup_{\| f \|_{L_x^2} \leq 1} \langle f, \int_t^{\infty} e^{isH} |u(s)|^2 u(s) ds \rangle_{L_{x}^2} \rightarrow 0,
\end{align*}
as $t \rightarrow \infty$. Note that 
\begin{align*}
\langle f, \int_t^{\infty} e^{isH} |u(s)|^2 u(s) ds \rangle_{L_{x}^2} &= \langle  e^{-isH} f, |u(s)|^2 u(s)  \rangle_{L_{t,x}^2 ([t, \infty) \times \R^3)} \\
&\leq \| e^{-itH} f \|_{L_{[t, \infty)}^3 L_x^{\frac {18}{5}}} \| u \|_{L_{[t, \infty)}^{\frac 92} L_x^{\frac {54}{13}}}^3 \\
&\lesssim \| f \|_{L_x^2} \| u \|_{L_{t,x}^4 ([t, \infty) \times \R^3) }^{\frac 83} \| u \|_{L_t^{\infty} L_x^6}^{\frac 13} .
\end{align*}
We used interpolation inequality $\| u \|_{L_t^{\frac 92} L_x^{\frac {54}{13}}} \leq \| u \|_{L_{t,x}^4}^{\frac 89} \| u \|_{L_t^{\infty} L_x^6}^{\frac 19}$. The last quantity converges to $0$ as $t \rightarrow \infty$ since $\| u \|_{L_{t,x}^4}$ is finite.  For $\dot H^1$, we need the following lemma.

\begin{lem}\label{lemma39}
For a solution $u$ of the given equation, $\| u \|_{L_t^3 L_x^9 ([0, \infty) \times \R^3)}$ is finite.
\end{lem}
\begin{proof}
By Sobolev embedding,
\be\label{l39}
\| u \|_{L_t^3 L_x^9 ([0, \infty) \times \R^3)}\lesssim \| |\nabla|^{\frac 12} u \|_{L_t^{3} L_x^{\frac {18}{5}}}.
\ee
We now subdivide $[0, \infty)$ into finitely many disjoint intervals $I_1, I_2, \cdots, I_M$ so that  
\begin{align*}
&\cup_{k=1}^M I_k = [0, \infty), \\
&\| u \|_{L_{I_k}^4 L_x^4} \leq \epsilon,\ \ 1 \leq k \leq M,
\end{align*}
for some $\epsilon >0$ which will be chosen later. On each interval $I_k=[a_k, b_k]$, we have
\begin{align*}
\| |\nabla|^{\frac 12} u \|_{L_t^{3} L_x^{\frac {18}{5}}} &\lesssim  \||\nabla|^{\frac 12} u(a_k) \|_{L_x^2} + \| |\nabla|^{\frac 12} (u \abs{u}^2) \|_{L_t^{\frac {54}{37}} L_x^{\frac {162}{115}}} \\
&\lesssim \| u(a_k) \|_{H^1_x} + \| |\nabla|^{\frac 12} u \|_{L_t^3 L_x^{\frac {18}{5}}} \| u \|_{L_t^3 L_x^9} \| u \|_{L_t^{54} L_x^{\frac {81}{26}}} \\
&\lesssim Q + \| |\nabla|^{\frac 12} u \|_{L_t^3 L_x^{\frac {18}{5}}}^2 \| u \|_{L_{t,x}^4}^{\frac {2}{27}} \| u \|_{L_t^{\infty} L_x^{\frac {150}{49}}}^{\frac {25}{27}} \\
&\lesssim  Q + \epsilon^{\frac 2{27}} \| |\nabla|^{\frac 12} u \|_{L_t^3 L_x^{\frac {18}{5}}}^2 \| u \|_{L_t^{\infty} L_x^{\frac {150}{49}}}^{\frac {25}{27}}\\
&\lesssim   Q + \epsilon^{\frac 2{27}} \| |\nabla|^{\frac 12} u \|_{L_t^3 L_x^{\frac {18}{5}}}^2\| u \|_{L_t^{\infty}H_x^{1}}^{\frac{25}{27}}.
\end{align*}
We take small enough $\epsilon$ to apply the continuity method. Note that $\epsilon$ only depends on the implicit constant of the Strichartz estimate and the size of the initial data. By the method of continuity, we conclude $\| |\nabla|^{\frac 12} u \|_{L_t^{3} L_x^{\frac {18}{5}}}$ is finite on each interval $I_k$. Since we have only finitely many intervals, $\| |\nabla|^{\frac 12} u \|_{L_t^{3} L_x^{\frac {18}{5}}}$ is finite on $[0, \infty) \times \R^3$, and the result follows by \eqref{l39}.
\end{proof}

Now for $\dot{H}^1$ we have
\[
\| \int_t^{\infty} e^{isH} |u(s)|^2 u(s) ds \|_{\dot H^1}\lesssim \| H^{\frac 12}\int_t^{\infty} e^{isH} |u(s)|^2 u(s) ds \|,
\]
and
\begin{align*}
\langle f, H^{\frac 12} \int_t^{\infty} e^{isH} |u(s)|^2 u(s) ds \rangle_{L_{x}^2} &= \langle  e^{-isH} f, H^{\frac 12} (|u(s)|^2 u(s))  \rangle_{L_{t,x}^2 ([t, \infty) \times \R^3)} \\
&\leq \| e^{-itH} f \|_{L_{[t, \infty)}^{\frac {20}9} L_x^5} \| H^{\frac 12} (|u|^2 u) \|_{L_{[t, \infty)}^{\frac {20}{11}} L_x^{\frac {5}{4}}} \\
&\lesssim \| f \|_{L_x^2} \| |\nabla| u \|_{L_{[t, \infty)}^{\infty} L_x^2} \| u \|_{L_{[t, \infty)}^{\frac {40}{11}} L_x^{\frac {20}{3}}}^2 \\
&\lesssim \| f \|_{L_x^2} \| |\nabla| u \|_{L_{[t, \infty)}^{\infty} L_x^2} \| u \|_{L_{[t, \infty)}^3 L_x^9}^{\frac {33}{20}} \| u \|_{L_{[t, \infty)}^{\infty} L_x^3}^{\frac {7}{20}} .
\end{align*}
We used interpolation inequality $\| u \|_{L_{[t, \infty)}^{\frac {40}{11}} L_x^{\frac {20}{3}}} \leq \| u \|_{L_{[t, \infty)}^3 L_x^9}^{\frac {33}{40}} \| u \|_{L_{[t, \infty)}^{\infty} L_x^3}^{\frac {7}{40}} $. Since $\| u \|_{L_t^3 L_x^9}$ is finite by Lemma \ref{lemma39}, the last quantity vanishes as $t \rightarrow \infty$ which completes the proof of scattering.

\appendix

\section{Failure of pointwise nonnegativity of $P_5$}
Let $x, y \in \R^3$ and $A$ be time independent, divergence-free.  Then the terms that appear in the integral in $P_5$ are
\[
\frac 12 \partial_{x_j}a(x,y)F_{kj}(x)p_{k}(x)  |u(y)|^2+\frac 14\partial_{x_j}a(x,y)F_{0j}(x)\abs{u(x)}^{2}  |u(y)|^2.
\]
Since $F =\ast \curl A,$ where $\ast$ is the Hodge star operator on forms, the above formula is
\begin{align}\label{ap1}
-\frac 12\curl A(x) \cdot \big(\nabla_x a(x,y) \times \vec p(x)\big) |u(y)|^2- \frac 14\nabla_{x}a(x,y)\cdot \nabla_{x}A_{0}(x)\abs{u(x)}^{2}  |u(y)|^2.
\end{align}
Note $\nabla_x a(x,y)$ is parallel to $x-y$ as long as $a(x,y) = a(|x-y|)$, so for any given $x=x_0$, we can find $y$ so that $\curl A(x_0) \cdot (\nabla_x a(x_0,y) \times \vec p(x_0)) >0$. Similarly, we can find $y$ so that $\nabla_{x}A_{0}$ and $x-y$ form an angle less than $\frac \pi 2$.
\newline\indent
Alternatively, we can write \eqref{ap1} as 
\[
 - \frac 12\vec p(x) \cdot (\curl A \times \nabla_{x}a) |u(y)|^2- \frac 14\nabla_{x}a(x,y)\cdot \nabla_{x}A_{0}(x)\abs{u(x)}^{2}  |u(y)|^2,
\]
and again as long as $a(x,y) = a(|x-y|)$, then this is a dot product of the momentum vector with a component of $\curl A$ tangent to the unit sphere centered at $y$ and the second term is the radial component of $\nabla_{x}A_{0}$ with respect to the sphere centered at $y$ (compare to the trapping component in \cite{FanelliVega09, JFA}).  Therefore, as we move $y$ around, pointwise nonnegativity in general is not possible.

\bibliography{im}
\bibliographystyle{plain}
%\vspace{.125in}

\end{document}